\documentclass[12pt,reqno]{amsart}
\usepackage{amsthm, amsfonts,  amsmath, amssymb}
\usepackage{url}
\usepackage[a4paper, top = 1.4in, bottom = 1.3in, left = 1.4in, right = 1.4in]{geometry}

\newtheorem{thm}{Theorem}[section]
\newtheorem{lem}[thm]{Lemma}
\newtheorem{prop}[thm]{Proposition}
\newtheorem{dfn}[thm]{Definition}

\newtheorem{remark}[thm]{Remark}
\newtheorem{cor}[thm]{Corollary}
\numberwithin{equation}{section}

\newcommand{\C}{\mathbb{C}}

\newcommand{\N}{\mathbb{N}}
\newcommand{\Q}{\mathbb{Q}}

\newcommand{\Z}{\mathbb{Z}}

\newcommand{\mcO}{\mathcal{O}}

\newcommand{\lra}{\longrightarrow}
\newcommand{\ra}{\rightarrow}

\newcommand{\mrm}[1]{\mathrm{#1}}

\newcommand{\psmat}[4]{\bigl( \begin{smallmatrix} #1 & #2 \\ #3 & #4 \end{smallmatrix} \bigr)}

\title[On sign changes of $q$-exponents of GMF's]{On sign changes of $q$-exponents of generalized modular functions}
\author{Narasimha Kumar}
\date{}
\email{narasimha.kumar@iith.ac.in}
\address{
Department of Mathematics \\
Indian Institute of Technology Hyderabad \\
Ordnance Factory Estate \\
Yeddumailaram 502205 \\
Andhra Pradesh, INDIA
}

\begin{document}
\begin{abstract}
Let $f$ be a generalized modular function (GMF) of weight $0$ on $\Gamma_0(N)$
such that its $q$-exponents $c(n)$($n \in \N$) are all real, and $\mrm{div}(f) = 0$.
In this note, we show the equidistribution of signs for $c(p) (p\ \mrm{prime})$
by using equidistribution theorems for normalized cuspidal eigenforms of integral weight.
\end{abstract}
\maketitle

\section{Introduction}
The problem of sign changes of Fourier coefficients of integral, half-integral weight modular forms, and of $q$-exponents of generalized modular functions
(GMF) has been studied quite extensively. For some results on sign changes for integral, half-integral weight modular forms, see \cite{KKP03},\cite{Koh07},\cite{BK08},\cite{HKKL12},\cite{IW},
and for generalized modular functions, see~\cite{KM08},~\cite{KM11}. In this note, we study equidistribution of signs for a family of 
$q$-exponents of generalized modular functions.

Let $f$ be a non-zero generalized modular function of weight $0$ on $\Gamma_0(N)$ (see Section~\ref{preliminaries} for the  definition). Then $f$
has an infinite product expansion 
\begin{equation}
\label{infinite-expansion}
 f(z) = c q^h \prod_{n = 1}^{\infty} (1-q^n)^{c(n)},
\end{equation}
where the product on the right-hand side of~$\eqref{infinite-expansion}$ is convergent in a small neighborhood of $q=0$, where
$q = e^{2 \pi i z}$. Here $c$ is a non-zero constant, $h$ is the order of $f$ at infinity, and the $c(n) (n \in \N)$ are uniquely
determined complex numbers~\cite{BKO04},~\cite{ES96}. 

In~\cite{KM08}, Kohnen and Martin constructed GMF's $f$ on $\Gamma_0(N)$ with $\mrm{div}(f) = 0$ 
such that the $q$-exponents $c(n)$($n \in \N$) take infinitely many different values.
Moreover, if $c(n)$($n \in \N$) are real, then they change signs infinitely often~\cite{KM11}, 
i.e., there are infinitely many $n$ such that $c(n)>0$ and there are infinitely many $n$ 
such that $c(n)<0$.

For prime exponents, Kohnen and Meher proved that the exponents $c(p)$ $(p\ \mrm{prime})$ take infinitely many different values, 
under some assumptions, see~\cite{KM11}. However, they have not considered the problem of sign changes for the $q$-exponents $c(p) (p\ \mrm{prime})$.
So, it is a natural question to ask about sign changes for the exponents $c(p) (p\ \mrm{prime})$,
assuming that they are real. 


In this note, we show that for a non-constant GMF $f$ of weight $0$ on $\Gamma_0(N)$ with $\mrm{div}(f) = 0$,
the  exponents $c(p)$($p$ prime) change signs infinitely often. Moreover, we will also study the equidistribution of signs 
for the  exponents $c(p)$($p$ prime) by using the Sato-Tate equidistribution theorem for normalized cuspidal eigenforms without complex
multiplication (CM)~\cite{BGHT11}, and Deuring's theorem for CM elliptic curves~\cite{Deu41}.

In the last section, we prove a general statement about the integrality of $q$-exponents $c(p)$($p$ prime) of 
generalized modular functions (GMF). As a consequence, we show that these functions are poorly behaved with respect to the integrality 
of these $q$-exponents.

\section{Preliminaries and Statements of results}	
\label{preliminaries} 
In this section, we will recall the definition of generalized modular functions and some basic results about them.
We refer the reader to the fundamental article~\cite{KM03} for more details.
\begin{dfn}
We say that $f$ is a generalized modular function (GMF) of integral weight $k$ on $\Gamma_0(N)$, if $f$ is a holomorphic 
function on the upper half-plane $\mathbb{H}$ and
$$ f \left(\frac{az+b}{cz+d} \right)= \chi(\gamma) (cz+d)^k f(z) \quad \forall \gamma=\psmat{a}{b}{c}{d} \in \Gamma_0(N)$$
for some (not necessarily unitary) character $\chi:\Gamma_0(N) \ra \C^{*}$.
\end{dfn}
We will also suppose that $\chi(\gamma) = 1$ for all parabolic $\gamma \in \Gamma_0(N)$ of trace $2$. 
We remark that in~\cite{KM03}, a GMF in the above sense was called as a parabolic GMF (PGMF).

If $f$ is a GMF of weight $0$ on $\Gamma_0(N)$ with $\mrm{div}(f) =  0$,
then its logarithmic derivative
\begin{equation}
\label{log-derivative}
 g:= \frac{1}{2\pi i} \frac{f^{\prime}}{f}
\end{equation}
is a cusp form of weight $2$ on $\Gamma_0(N)$ with trivial character (cf. loc. cit.). Conversely, if one starts with such a cusp form $g$, then there
exists a GMF $f$ of weight $0$ on $\Gamma_0(N)$ such that~\eqref{log-derivative} is satisfied and $f$ is uniquely determined up to a non-zero scalar~\cite{KM03}. 
Now, we will make some remarks. 
\begin{remark}
The cardinality of $\Gamma_0(N)$-inequivalent zeros of $f^{\prime}$, the derivative of $f$, depends on whether the weight $k=0$ or not. More precisely:
\begin{itemize}
 \item $f^{\prime}$ has only finitely many $\Gamma_0(N)$-inequivalent zeros if $k=0$ $($by~\eqref{log-derivative}$)$, where as
\item  $f^{\prime}$ has infinitely many $\Gamma_0(N)$-inequivalent zeros if $k>0$~\cite[Remark $5.4$]{SS12}.  
\end{itemize}
 \end{remark}

Suppose the Fourier expansion of $g(z)$ is given by
\begin{equation*}
g(z) = \sum_{n = 1}^{\infty} b(n) q^n.
\end{equation*}

\begin{remark}
\label{key-equation}
The field generated by the Fourier coefficients  of $g$ over $\Q$ is equal to the field generated by the
$q$-exponents of $f$ over $\Q$, where $f,g$ are related by~\eqref{log-derivative}. This follows from the equalities
$b(n) = -{\sum}_{d|n} dc(d)$ and $nc(n)=- {\sum}_{d|n}\mu(d) b(n/d)$ $(n \geq 1)$.
\end{remark}

Throughout this note, we shall assume that the modular form $g$ is a normalized cuspidal eigenform, and $c(p)(p\ \mrm{prime})$ are real numbers. 
Let $\mathbb{P}$ denote the set of all prime numbers and $p$ always denote a prime number, unless explicitly stated otherwise. 
We recall the notion of natural density for subsets of $\mathbb{P}$.
\begin{dfn}
Let $S$ be a subset of  $\mathbb{P}$. The set $S$ has  natural density $d(S)$
if the limit
\begin{equation}
\underset{x \ra \infty}{\mrm{lim}}\  \frac{\# \{ p \leq x: p\in S\}}{\# \{ p \leq x: p\in \mathbb{P}\}}
\end{equation}
exists and is equal to $d(S)$. 
\end{dfn}
\begin{remark}
$d(S)=0$ if $|S|<\infty$.
\end{remark}
For notational convenience, we let $\mathbb{P}_{<0}$ denote 
the set $\{ p \in \mathbb{P}: p\nmid N,\ c(p)<0 \}$, and similarly $\mathbb{P}_{>0}$, $\mathbb{P}_{\leq 0}$, $\mathbb{P}_{\geq 0}$, 
and $\mathbb{P}_{=0}$. Now we can state the main results of this note. 
\begin{thm}[Non-CM case]
\label{main-thm}
Let $f$ be a non-constant GMF of weight $0$ on $\Gamma_0(N)$ with $div(f) = 0$.
Suppose that $g$ $($as in~\eqref{log-derivative}$)$ is a normalized Hecke eigenform without 
complex multiplication (CM). Then the exponents $c(p)$ $(p$\ prime$)$ change signs infinitely often.
Moreover, the sets $$\mathbb{P}_{>0}, \mathbb{P}_{<0}, \mathbb{P}_{\geq 0}, \mathbb{P}_{\leq 0}$$
have natural density $1/2$,
and $d(\mathbb{P}_{=0}) = 0$.
\end{thm}
We will also prove a stronger version of this theorem (cf. Theorem~\ref{main-thm-finer} in the text)
by using a hybrid Chebotarev-Sato-Tate theorem~\cite[Chap. 12]{MM12}. 

The above theorem can be thought of as an analogue, for generalized modular functions, of Theorem $5.1$ in~\cite{IW}, where Inam and Wiese have studied 
equidistribution of signs for some families of coefficients of half-integral weight modular eigenforms,
under the assumption that their Shimura lifts are cuspidal eigenforms of non-CM type. 

In the CM case, Prof. G. Wiese had kindly informed us that, he and his collaborators are working on extending the results in~\cite{IW} 
to the CM Shimura lifts as well. However, in our situation, we are able to study the sign changes of $q$-exponents $c(p)$ ($p$ prime), 
when $g$ (as in~\eqref{log-derivative}) is a CM form. More precisely:

\begin{thm}[CM case]
\label{main-thm-2}
Let $f$ be a non-constant GMF of weight $0$ on $\Gamma_0(N)$ with $div(f) = 0$. 
Suppose that the normalized Hecke eigenform $g\ ($as in~\eqref{log-derivative}$)$
corresponds to a CM elliptic curve $E$ over $\Q$. 
Then the exponents $c(p)(p$ prime$)$ change signs infinitely often. 
Moreover, the sets $$\mathbb{P}_{>0}, \mathbb{P}_{< 0}$$
have natural density $\frac{3}{4}$, $\frac{1}{4}$, respectively, 
and $d(\mathbb{P}_{=0}) = 0$.
\end{thm}
For the proof of Theorem~\ref{main-thm-2}, we use Deuring's equidistribution theorem for CM elliptic curves over $\Q$.
It is widely believed that a similar equidistribution theorem also holds for CM modular forms, but we could not locate a 
reference for it in the literature. Hence, we stated the theorem only for CM elliptic curves, otherwise the rest of the proof is same in both cases.

Finally, observe that the natural density of the set $\mathbb{P}_{<0}$ (or $\mathbb{P}_{>0}$) depends on whether $g$ has CM or not.
In \S\ref{section-main-thm} and \S\ref{section-main-thm-2}, we prove Theorem~\ref{main-thm} and Theorem~\ref{main-thm-2}, respectively.

\section{Proof of Theorem~\ref{main-thm}}
\label{section-main-thm}
In this section, we shall prove Theorem~\ref{main-thm} by using the Sato-Tate equidistribution for normalized cuspidal eigenforms of integral weight. Before we proceed to the proof, let us recall the Sato-Tate measure 
and the Sato-Tate equidistribution theorem~\cite{BGHT11}. 

\begin{dfn}
The Sato-Tate measure $\mu_{\mrm{ST}}$ is the probability measure on  $[-1,1]$ given by $\frac{2}{\pi} \sqrt{1-t^2} dt$. 
\end{dfn}

Let $g = \sum_{n=1}^{\infty} b(n)q^n$ be a normalized cuspidal eigenform of weight $2k$ on $\Gamma_0(N)$.
By Deligne's bound, one has 
$$ |b(p)| \leq 2 p^{k-\frac{1}{2}}, $$
and we let 
\begin{equation}
\label{Sato-Tate-normalization}
B_k(p):=\frac{b(p)}{2p^{k-\frac{1}{2}}} \in [-1,1].
\end{equation}
By~\cite[Thm. B.]{BGHT11}, we have the following Sato-Tate equidistribution theorem for the cuspidal eigenform $g$. 
\begin{thm}[Barnet-Lamb, Geraghty, Harris, Taylor]
\label{Sato-Tate-thm}
Let $k \geq 1$ and let $g$ be a normalized cuspidal eigenform of weight $2k$ on $\Gamma_0(N)$
without complex multiplication. Then the numbers $\{ B_k(p) \}_{\{p \nmid N\}}$ are equidistributed
in $[-1,1]$ with respect to the Sato-Tate measure $\mu_{\mrm{ST}}$.
\end{thm}

\begin{cor}
\label{Sato-Tate-cor}
Let $g$ be as in the above Theorem. For any subinterval $I \subseteq [-1,1]$, we have 
$$ \underset{x \ra \infty}{\mathrm{lim}}  \ \frac{\pi_I(x)}{\pi(x)} = \mu_{\mrm{ST}}(I) = 
\frac{2}{\pi} \int_I  \sqrt{1-t^2}  dt,$$
where $\pi_I(x):= \# \left\lbrace p\leq x: p\nmid N,B_k(p) \in I \right\rbrace$.
\end{cor}

We let $\pi_{<0}(x)$ denote the number 
$ \# \{p\leq x: p \in \mathbb{P}_{<0}\}$, and similarly for 
$\pi_{>0}(x), \pi_{\leq 0}(x)$, $\pi_{\geq 0}(x)$, and $\pi_{=0}(x)$. 

Now we start the proof of Theorem~\ref{main-thm}. Since $g$ is the logarithmic derivative of $f$, we have: 
$$ b(n) = - \sum_{d|n} dc(d) \quad (n\geq 1). $$
Since $c(1)=-1$, we have $b(p) = 1-pc(p).$
Therefore,
$ \frac{b(p)}{2\sqrt{p}} =  \frac{1}{{2\sqrt{p}}}-\frac{c(p)\sqrt{p}}{2}$ and hence
\begin{equation*}
c(p)>0 \Longleftrightarrow  -1 \leq B_1(p) < \frac{1}{2 \sqrt{p}}, \quad c(p)<0 \Longleftrightarrow \frac{1}{2 \sqrt{p}} < B_1(p) \leq 1, 
\end{equation*}
where $B_1(p) = \frac{b(p)}{2\sqrt{p}}$ by definition (see~\eqref{Sato-Tate-normalization}).

\begin{lem}
\label{lem-main-thm}
$\underset{x \ra \infty}{\mrm{lim\  sup}} 
\ \frac{\pi_{< 0}(x)}{\pi(x)} \leq \mu_{\mrm{ST}}([0,1]) = \frac{1}{2}.$
Similarly,  $\underset{x \ra \infty}{\mrm{lim\  sup}} \ \frac{\pi_{\leq 0}(x)}{\pi(x)} \leq \frac{1}{2}.$
\end{lem}
\begin{proof}
We have the following inequality $$\# \{p\leq x: p \nmid N,\ c(p)<0 \} \leq \# \{ p \leq x: p \nmid N,\ B_1(p) \in [0,1] \}.$$
Now divide the above inequality by $\pi(x)$ and make $x$ tends to $\infty$, then the lemma follows from Corollary~\ref{Sato-Tate-cor}.
\end{proof} 

\begin{prop}
\label{prop-main-thm}
 $\underset{x \ra \infty}{\mrm{lim\ inf}} \ \frac{\pi_{< 0}(x)}{\pi(x)} \geq \mu_{\mrm{ST}}([0,1])=\frac{1}{2}.$
Similarly, $\underset{x \ra \infty}{\mrm{lim\ inf}} \ \frac{\pi_{\leq 0}(x)}{\pi(x)} \geq \frac{1}{2}.$ 
\end{prop}
\begin{proof}
For any fixed (but small) $\epsilon >0$, we have the following inclusion of sets 
$$ \left\lbrace p\leq x: p \nmid N,\  c(p)<0 \right\rbrace \supseteq \left\lbrace  p \leq x: p \nmid N,\ p> \frac{1}{4\epsilon^2},\ B_1(p) \in [\epsilon,1] \right\rbrace.$$ 
Hence, we have
$$ \# \{p\leq x: p \nmid N,\ c(p)<0\} + \pi\left(\frac{1}{4\epsilon^2}\right) \geq \# \{ p \leq x:  p \nmid N,\  B_1(p) \in [\epsilon,1] \}.$$
Now divide the above inequality by $\pi(x)$
$$ \frac{\# \{p\leq x:  p \nmid N,\ c(p)<0\}}{\pi(x)} + \frac{\pi\left(\frac{1}{4\epsilon^2}\right) }{\pi(x)} 
\geq \frac{\# \left\lbrace p \leq x: p \nmid N,\  B_1(p) \in [\epsilon,1] \right\rbrace}{\pi(x)}.$$
The term $\frac{\pi(\frac{1}{4\epsilon^2})}{\pi(x)}$ tends to zero as $x\ra \infty$ as $\pi(\frac{1}{4\epsilon^2})$ is finite.
By Corollary~\ref{Sato-Tate-cor}, we have
$$ \frac{\# \{ p\leq x : p \nmid N,\ B_1(p) \in [\epsilon,1]    \} }{\pi(x)}  \lra \mu_{\mrm{ST}}([\epsilon, 1]) \quad\ \mrm{as}\quad 	x \ra \infty. $$ 
This implies that 
\begin{equation}
\label{key-inequality}
\underset{x \ra \infty}{\mrm{lim\ inf}} \ \frac{\pi_{<0}(x)}{\pi(x)} \geq \mu_{\mrm{ST}}([\epsilon,1]),
\end{equation}
where $\pi_{<0}(x)=\# \{p\leq x:p \nmid N,\ c(p)<0\}$ by definition. Since the inequality~\eqref{key-inequality} holds for all $\epsilon>0$, we have that 
$$ \underset{x \ra \infty}{\mrm{lim\  inf}} \ \frac{\pi_{<0}(x)}{\pi(x)} \geq \mu_{\mrm{ST}}([0,1]) = \frac{1}{2}.$$
A similarly proof  shows that
$\underset{x \ra \infty}{\mrm{lim\ inf}} \ \frac{\pi_{\leq 0}(x)}{\pi(x)} \geq \frac{1}{2}.$
 \end{proof}

\begin{proof}[\textbf{Proof of Theorem~\ref{main-thm}}]
By Lemma~\ref{lem-main-thm} and Proposition~\ref{prop-main-thm}, we have
$$ \frac{1}{2} \geq \underset{x \ra \infty}{\mrm{lim\  sup}} \ \frac{\pi_{< 0}(x)}{\pi(x)} \geq \underset{x \ra \infty}{\mrm{lim\  inf}} 
\ \frac{\pi_{< 0}(x)}{\pi(x)} \geq   \frac{1}{2},$$
Hence, the limit  $ \underset{x \ra \infty}{\mrm{lim}} \frac{\pi_{<0}(x)}{\pi(x)}$ exists 
and is equal to $\frac{1}{2}$. Therefore, the natural density of the set $\mathbb{P}_{<0}$ is $\frac{1}{2}$. 
Similarly, one can also argue for the sets $\mathbb{P}_{>0}$, $\mathbb{P}_{ \leq 0}$, and $\mathbb{P}_{ \geq 0}$, 
and show that the natural densities of these sets are $\frac{1}{2}$. The claim for $\mathbb{P}_{=0}$ follows from 
the former statements.
\end{proof}

By arguing as in the proof of the above theorem, one can also prove the following statement.
\begin{thm}
Let $f$ be a GMF such that the hypotheses of Theorem~\ref{main-thm} hold. Then, for any subinterval $I \subseteq [-1,1]$, the natural density of 
the set $$\left\lbrace p \in \mathbb{P} : p \nmid N,\ c_1(p) \in I \right\rbrace$$
is $\mu_{\mrm{ST}}(I)$, where $c_1(p):= \frac{c(p)\sqrt{p}}{2}$ is the Sato-Tate normalization for $f$. 
\end{thm}

We finish this section with a stronger version of Theorem~\ref{main-thm}. 
\begin{thm}
\label{main-thm-finer}
Let $f$ be a GMF such that the hypotheses of Theorem~\ref{main-thm} hold. Let $q$ be a natural number and $a$ an integer with $(a, q)=1$.
Then, for any subinterval $I \subseteq [-1,1]$, the natural density of primes $p$ for which $c_1(p) \in I$ and
$p \equiv a \pmod q$ is ${\mu_{\mrm{ST}}(I)}/{\varphi(q)}$, i.e.,
$$ \underset{x \ra \infty}{\mathrm{lim}}\   \frac{\pi_I(x)}{\pi(x)} = \frac{\mu_{\mrm{ST}}(I)}{\varphi(q)} 
= \frac{2}{\pi\varphi(q)} \int_I  \sqrt{1-t^2}  dt,$$	
where $\pi_I(x):= \# \left\lbrace p\leq x: p\nmid N,c_1(p) \in I\ \mrm{and}\ p \equiv a \pmod q\right\rbrace$.
In particular, the natural density of the set
$\left\lbrace p \in \mathbb{P}:p \nmid N,c(p)>0 \ \mrm{and}\ p \equiv a \pmod q \right\rbrace$ is
$\frac{1}{2 \varphi(q)}$. 
\end{thm}
\begin{proof}
The proof is similar to the proof of Theorem~\ref{main-thm}, except that here we use the hybrid 
Chebotarev-Sato-Tate theorem~\cite[Cor. 1.2 in Chap. 12]{MM12} in the proof. 
\end{proof}

\section{Proof of Theorem~\ref{main-thm-2}}
\label{section-main-thm-2}
In this section, we shall prove Theorem~\ref{main-thm-2} by using  Deuring's equidistribution theorem for CM elliptic curves.   
We assume that $g$ is the associated eigenform of weight $2$ on $\Gamma_0(N)$ to  a CM elliptic curve $E$ over $\Q$ of conductor $N$,
in the sense of~\cite{Wil95},~\cite{BCDT01}.
Recall that, $g(z) = \sum_{n=1}^{\infty} b(n) q^n$ with $b(1)=1$.

Before we proceed to the proof, let us recall the appropriate measure and the  equidistribution theorem for CM elliptic curves.
By ~\cite{Deu41}, the appropriate measure $\mu_{\mrm{CM}}$ for equidistribution of the quantity $B_1(p):= b(p)/2\sqrt{p}\in [-1,0) \cup (0,1]$
is $\frac{1}{2\pi} (1-t^2)^{-\frac{1}{2}}$. At the real number $0$, there should be a Dirac measure with weight $1/2$. More precisely:
\begin{thm}[Deuring]
\label{Deuring}
Suppose that $K$ is an imaginary quadratic field and that $E$ has CM by an order in $K$. Then
for any prime $p$ of good reduction for $E$, we have 
$$b(p)=0 \Longleftrightarrow p \ is\ inert\  in\  K.$$ 
Furthermore, if $I \subset [-1,1]$ is some interval with $0 \not \in I$, then 
$$ \underset{x \ra \infty}{\mathrm{lim}} \ \frac{\pi_I(x)}{\pi(x)} = \mu_{\mrm{CM}}(I) = \frac{1}{2\pi}
\int_I \frac{1}{\sqrt{1-t^2}} dt,$$
where $\pi_I(x):= \# \left\lbrace  p\leq x:  p \nmid N, \ B_1(p) \in I \right\rbrace$.
\end{thm}


Now, we shall make some remarks. 
\begin{remark}
\label{CM-remar}
By Theorem~\ref{Deuring}, the natural density of primes $p$ for which $b(p)=0$
is equal to the natural density of inert primes $p$ in $K$. By the Chebotarev density theorem, the natural density 
of inert primes $p$ in $K$ is $1/2$.
Observe that $b(p) = 0 \Longleftrightarrow c(p)=\frac{1}{p}>0$. 
\end{remark}
\begin{remark}
\label{Integration-remark}
$$\mu_{\mrm{CM}}([-1,0))=\mu_{\mrm{CM}}((0,1])=\frac{1}{2\pi}\int_0^1 \frac{dx}{\sqrt{1-x^2}} = \frac{1}{\pi} \int_0^1 \sqrt{1-x^2} \ dx = \frac{1}{4}.$$
\end{remark}

Now we are ready to prove Theorem~\ref{main-thm-2}.  
\begin{lem}
\label{lem-main-thm-2}
$\underset{x \ra \infty}{\mrm{lim\ inf}} 
\ \frac{\pi_{> 0}(x)}{\pi(x)} \geq \frac{1}{2}+\mu_{\mrm{CM}}((0,1])$.
Similarly, $\underset{x \ra \infty}{\mrm{lim\  sup}} 
\ \frac{\pi_{< 0}(x)}{\pi(x)} \leq  \frac{1}{4}.$
\end{lem}

\begin{prop}
\label{prop-main-thm-2}
$\underset{x \ra \infty}{\mrm{lim\ sup}} 
\ \frac{\pi_{> 0}(x)}{\pi(x)} \leq  \frac{1}{2}+\mu_{\mrm{CM}}((0,1]).$
Similarly, 
 $\underset{x \ra \infty}{\mrm{lim\ inf}} \ \frac{\pi_{< 0}(x)}{\pi(x)} \geq \frac{1}{4}.$
\end{prop}

The proofs of Lemma~\ref{lem-main-thm-2} and Proposition~\ref{prop-main-thm-2} are
similar to the proofs of Lemma \ref{lem-main-thm} and Proposition~\ref{prop-main-thm}, except that
here we use Theorem~\ref{Deuring} instead of Theorem~\ref{Sato-Tate-thm} and also the Remark~\ref{CM-remar}
in the proof.
 
By Lemma~\ref{lem-main-thm-2} and by Proposition~\ref{prop-main-thm-2}, 
we have
$$\frac{3}{4} \leq \underset{x \ra \infty}{\mathrm{lim\ inf}} \ \frac{\pi_{>0} (x)}{\pi(x)} 
  \leq \underset{x \ra \infty}{\mathrm{lim\ sup}} \ \frac{\pi_{>0} (x)}{\pi(x)} \leq  \frac{3}{4}. $$
Hence, the limit  $\underset{x \ra \infty}{\mrm{lim}} \frac{\pi_{>0}(x)}{\pi(x)}$
exists, and is equal to $\frac{3}{4}$, i.e., $d(\mathbb{P}_{>0})=\frac{3}{4}$. 
Similarly,  we can show that the set $\mathbb{P}_{<0}$ has natural density $\frac{1}{4}$, i.e., $d(\mathbb{P}_{<0}) = \frac{1}{4}$.
This finishes the proof of Theorem~\ref{main-thm-2}.

We finish this section with the following assertion whose proof is similar to the proof of the above theorem,
of course, which itself is an application of Theorem~\ref{Deuring}.

\begin{thm}
Let $f$ be a GMF such that the hypotheses of Theorem~\ref{main-thm-2} hold. If $I \subset [-1,1]$ is either a subinterval of $[-1,0]$ or some closed interval in $(0,1]$, 
then 
\begin{equation}
\label{final-eqn}
\underset{x \ra \infty}{\mathrm{lim}}  \frac{\# \{  p\leq x:  p \nmid N, \ c(p)\sqrt{p}/2 \in I \}}{\pi(x)} =  \frac{1}{2\pi}
\int_I \frac{1}{\sqrt{1-t^2}} dt,
 \end{equation}
i.e., the natural density of the set $$\{p\in \mathbb{P}: p\nmid N,\ c(p)\sqrt{p}/2 \in I \}$$
is $\mu_{\mrm{CM}}(I)$. 
\end{thm}

In the above Theorem, there is a reason for restricting ourselves to the closed intervals of $(0,1]$. For example, if we take $I=(0,1]$, then 
the limit on the left side of the equality in~\eqref{final-eqn} is $3/4$ (by Theorem~\ref{main-thm-2}), where as the quantity on the right side of the equality is  $1/4$ 
(by Remark~\ref{Integration-remark}). Therefore, if we restrict ourselves to closed intervals of $(0,1]$, then 
there is no contribution of primes coming from the inert primes $p$ of $K$ while calculating these densities (cf. Remark~\ref{CM-remar}). 

By the same reasoning, we see that there is no harm in allowing $I$ to be a subinterval of 
$[-1,0]$ rather than only of $[-1,0)$ (like in Theorem~\ref{Deuring}).

\section{On the integrality of $q$-exponents of GMF}
In this penultimate section, we prove a general statement about the integrality of the $q$-exponents $c(p)$($p$ prime) of 
generalized modular functions (GMF). As a consequence, we also show that these functions are poorly behaved with respect to the integrality 
of these $q$-exponents.

For a generalized modular function $f$, we let $K_f$ to denote the number field generated 
by the $q$-exponents $c(n) (n \in \N)$ of $f$ and $\mcO_{K_f}$ to denote the integral closure of $\Z$ in $K_{f}$.

\begin{thm}
\label{main-thm-2-application}
Let $f$ be a non-constant GMF of weight $0$ on $\Gamma_0(N)$ with $div(f) = 0$. 
Suppose that $g$ $($as in ~\eqref{log-derivative}$)$ is a normalized Hecke eigenform.
Then, $c(p)$ is non-zero and integral only for finitely many primes $p$.
\end{thm}
\begin{proof}

Assume the contrary, i.e., there exists infinitely many primes $p$ such that $c(p)$ is non-zero and integral.
Let $S$ denote that infinite set of primes. 

Recall that $g(z) = \sum_{n=1}^{\infty} b(n) q^n$ with $b(1)=1$. For $p \in S$, $c(p) \in \mathcal{O}_{K_f}-\{0\}$. 
Therefore, $p$ divides $b(p)-1$ in $\mathcal{O}_{K_f}$ for all $p \in S$. Now, we show that this cannot happen.

Suppose $[K_f:\Q]=d$.
For $p \in S$, if $p$ divides $b(p)-1$ in $\mathcal{O}_{K_f}$, then 
$$ p^d \leq |N_{K_f/\Q}(b(p)-1)| \leq  (2\sqrt{p}+1)^d,$$
where $N_{K_f/\Q}$ is the Norm map from $K_f$ to $\Q$. The latter inequality follows from $|b(p)| \leq 2 \sqrt{p}$. 
In particular, this implies that, for all $p \in S$,
\begin{equation}
\label{integral-exponents}
 \left( \frac{p}{\sqrt{p}+1} \right)^d  \leq 2^d. 
\end{equation}
The above inequality can only hold for finitely many primes, which is a contradiction to 
the fact that $S$ is infinite.
\end{proof}
Recall that, $\mathbb{P}_0=\{p \in \mathbb{P} | c(p) =0 \}$. The above theorem implies that, if the set $\mathbb{P}_0$ is finite, 
then $c(p)$ is integral only for finitely many primes $p$. But, we only know that the density of $\mathbb{P}_0$ is zero, 
hence we have: 
\begin{cor}
\label{main-cor-2}
Let $f$ be a non-constant GMF of weight $0$ on $\Gamma_0(N)$ with $div(f) = 0$. 
Then $c(p) \not \in \mathcal{O}_{K_f}$ for a density $1$ set of primes $p \in \mathbb{P}$.
In particular, if all prime $q$-exponents $c(p)$ are in $\mcO_{K_f}$,
then $f=1$ is constant.
\end{cor}
\begin{proof}
By Theorem~\ref{main-thm} and Theorem~\ref{main-thm-2}, we see that the density of primes $p$
for which $c(p)=0$ is zero, i.e., $c(p)$ is non-zero for a density $1$ set of primes $p \in \mathbb{P}$.
Now, the corollary follows follows from the above Theorem.
\end{proof}

We finish this article with the following Corollary, which gives an another proof of~\cite[Thm. 1]{KM08-N}.
\begin{cor}[Kohnen-Mason]
\label{exponent-thm-2}
Let $f =\sum_{n=0}^{\infty} a(n) q^n$ be a GMF of weight $0$ and level $N$ with $div(f) = 0$.
Suppose $g$ $($as in~\eqref{log-derivative}$)$ is a normalized Hecke eigenform. 
If $a(0)=1$ and $a(n)\in \Z$ for $n \in \N$, then $f = 1$ is constant.
\end{cor}
\begin{proof}
Since $K_f=\Q$, one can show that if $a(0)=1$ and $a(n)\in \Z$ for all $n \in \N$, 
then $c(n) \in \Z$ for all $n \in \N$ (cf. the proof of~\cite[Thm. 1]{KM08-N}).
In particular, $c(p) \in \Z$ for all $p$, which is a contradiction to the above Corollary.
\end{proof}

\section{Acknowledgments}
The author would like to thank the Mathematics Center Heidelberg  for supporting his postdoctoral stay at Heidelberg,
where most of this work was carried out, Prof. G. B\"ockle for making some useful comments on earlier drafts, and 
Prof. W. Kohnen for providing the manuscripts~\cite{KM08},~\cite{KM11}. 

\bibliographystyle{plain, abbrv}

\end{document}